\documentclass[12pt,a4paper]{article}

\usepackage[english]{babel}
\usepackage[utf8]{inputenc}
\usepackage{amsmath,amssymb,amsthm,thmtools,xcolor,wasysym,xfrac,calc,indentfirst,enumerate}
\usepackage[all]{xy}
\makeatletter
\newcommand{\oset}[3][0ex]{%
	\mathrel{\mathop{#3}\limits^{
			\vbox to#1{\kern-2\ex@
				\hbox{$\scriptscriptstyle#2$}\vss}}}}
\makeatother
\newcommand\Ve{\ensuremath{\varepsilon}}
\newcommand{\Pmc}{\leftharpoonup}
\newcommand\Ipmc{{{\,}{}_\imath\!\!\leftharpoonup{ }}}
\newcommand{\Hom}{\mathop{\text{Hom}}}
\newcommand{\Pmci}{{\,}{\oset{\hspace*{3pt}\lambda\hspace*{-1.5pt}{\text{\textquotesingle}}}{\Pmc}}{\,}}
\newcommand{\Pma}{\rightarrow}
\newcommand{\Ipma}{\ensuremath{\Pma_\imath}}
\newcommand{\Pmai}{{\,}{\oset{\lambda\hspace*{-1.5pt}\text{\textquotesingle}\hspace*{1pt}}{\Pma}}{\,}}

\newcommand{\Pcci}{\lambda'_{\Pmc}}
\let\Mc\blacktriangleleft
\newcommand{\Ma}{\triangleright}
\newcommand{\U}{{}^{-1}{}}
\newcommand{\Up}{{}^{-\!\overline{\hspace*{1pt}1}}{}}
\newcommand{\Z}{{}^{0}{}}
\newcommand{\Zp}{{}^{\overline{0}}{}}
\newcommand{\Ti}{{\theta^{-1}}}
\let\longto\longrightarrow

\newcommand{\Not}[2][0]{
	\setcounter{enumi}{#1}
	\renewcommand{\theenumi}{#2\arabic{enumi}}
	\renewcommand{\labelenumi}{(\theenumi)}
	\setlength{\itemindent}{\widthof{#2}}
	\setlength{\itemsep}{4pt}
}
\newcommand{\NNot}[2][0]{
	\setcounter{enumii}{#1}
	\renewcommand{\theenumi}{}
	\renewcommand{\theenumii}{#2\arabic{enumii}}
	\renewcommand{\labelenumii}{(\theenumii)}
	\setlength{\itemindent}{\widthof{#2}}
	\setlength{\itemsep}{4pt}
}

\RequirePackage{mathtools}
\allowdisplaybreaks %Habilita a quebra de página em equações...
\newlength{\IgD}
\newcommand{\REf}[1]{ %Para 
	&\settowidth{\IgD}{~\scriptsize (\ref{#1})~}\hspace*{.5\IgD}&\mathclap{\stackrel{(\ref{#1})}{=}}\settowidth{\IgD}{~\scriptsize (\ref{#1})~}\hspace*{.5\IgD}&&
}
\newcommand{\EREf}{ %Para 
	&\phantom{=}&\mathclap{=}\phantom{=}&&
}
\usepackage{hyperref}

\begin{document}\frenchspacing

\title{Globalizations for partial (co)actions on coalgebras}%Paper title
\author{Felipe Castro and Glauber Quadros%
	\thanks{The authors were partially supported by CNPq, Brazil. The authors would like to thank Antonio Paques, Alveri Sant'Ana and the referee, whose comments, corrections and suggestions were very useful to improve the manuscript. We would like to thanks Lourdes Haase for her corrections and suggestions about the paper writing.\newline%
	\indent\textbf{2010 Mathematics Subject Classification:} primary 16T15; secondary 16T99, 16W22\newline%
	\indent\textbf{Key words and phrases:} Hopf algebras, partial action, partial coaction, globalization, partial module coalgebra, partial comodule coalgebra.%
	}%
}%Authors names
%\shorttitle{Globalization for partial (co)actions}%Short paper title for the headers

%\shortauthor{F.~Castro, G.~Quadros}%Short authors names for the headers
%\address[F.~Castro]{Departamento de Matemática\\ Universidade Federal de Santa Catarina\\88040-900 Brazil}{f.castro@ufsc.br}{http://mtm.ufsc.br/{\textasciitilde}fcastro}
%\address[G.~Quadros]{Coordenadoria Acadêmica\\Universidade Federal de Santa Maria\\96506-322 Brazil}{glauber.quadros@ufsm.br}{}

%\communicated{R. Wisbauer}

%\research% Type of the paper. \survey command may be used also.

%\date{July 14, 2016}
\date{May 8, 2019}

%%%%%%%%%%%%%%%%%% TEO %%%%%%%%%%%%%%%%%%%%%
\theoremstyle{plain}
	\newtheorem{teo}{Theorem}[section]
	\newtheorem{prop}[teo]{Proposition}
	\newtheorem{lema}[teo]{Lemma}
	\newtheorem{coro}[teo]{Corollary}
	\newtheorem{dfn}[teo]{Definition}
\theoremstyle{definition}
	\newtheorem{obs}[teo]{Remark}
	\newtheorem{ex}[teo]{Example}

\maketitle

%%%%%%%%%%%%%%%%%%%%%%%%%%%%%%%%%%%%%%%%%%%%%%%%%%%%%%%%%%%%%%%%%%%%%%%%%%%%%%%%%%%%%%%%%%%%
%%%%%%%%%%%%%%%%%%%%%                                %%%%%%%%%%%%%%%%%%%%%%%%%%%%%%%%%%%%%%%
%%%%%%%%%%%%%%%%%%%%%            ABSTRACT            %%%%%%%%%%%%%%%%%%%%%%%%%%%%%%%%%%%%%%%
%%%%%%%%%%%%%%%%%%%%%                                %%%%%%%%%%%%%%%%%%%%%%%%%%%%%%%%%%%%%%%
%%%%%%%%%%%%%%%%%%%%%%%%%%%%%%%%%%%%%%%%%%%%%%%%%%%%%%%%%%%%%%%%%%%%%%%%%%%%%%%%%%%%%%%%%%%%

\begin{abstract}
	In this paper, we introduce the notion of globalization for partial module coalgebra and for partial comodule coalgebra. We show that every partial module coalgebra is globalizable exhibiting a standard globalization. We also show the existence of globalization for a partial comodule coalgebra, provided a certain rationality condition. Moreover, we show a relationship between the globalization for the (co)module coalgebra and the usual globalization for the (co)module algebra.
\end{abstract}

%\subjclass{2010}{primary 16T15; secondary 16T99, 16W22}

%\keywords{Hopf algebras, partial action, partial coaction, globalization, partial module coalgebra, partial comodule coalgebra.}

%%%%%%%%%%%%%%%%%%%%%%%%%%%%%%%%%%%%%%%%%%%%%%%%%%%%%%%%%%%%%%%%%%%%%%%%%%%%%%%%%%%%%%%%%%%%
%%%%%%%%%%%%%%%%%%%%%                                %%%%%%%%%%%%%%%%%%%%%%%%%%%%%%%%%%%%%%%
%%%%%%%%%%%%%%%%%%%%%           INTRODUÇÃO           %%%%%%%%%%%%%%%%%%%%%%%%%%%%%%%%%%%%%%%
%%%%%%%%%%%%%%%%%%%%%                                %%%%%%%%%%%%%%%%%%%%%%%%%%%%%%%%%%%%%%%
%%%%%%%%%%%%%%%%%%%%%%%%%%%%%%%%%%%%%%%%%%%%%%%%%%%%%%%%%%%%%%%%%%%%%%%%%%%%%%%%%%%%%%%%%%%%
\section{Introduction}

Partial actions of groups were first considered in the context of operators algebra (cf.~\cite{E}). Dokuchaev and Exel in \cite{DE} introduced partial actions of groups in a purely algebraic context, obtained several classical results in the setting of partial actions of groups and covering the actions of groups. The actions of Hopf algebras on algebras also generalize the theory of actions of groups (cf.~\cite[Example~4.1.6]{M}). The notion of action of groups was extended in two directions, in both contexts extensive theories were developed, where we highlight the Morita and Galois theories. As a natural task, Caenepeel and Janssen in \cite{CJ} introduced the concept of partial actions of Hopf algebras on algebras, called partial module algebras. This new theory arose unifying all the above theories.

The globalization process, which consist in constructing an action of a group such that a given partial action can be seen inside it, was first described by Abadie (cf. \cite{A}) and it plays an important role in the context of partial actions of groups. Alves and Batista successfully translated it to the new context of partial module algebras (cf. \cite{AB1,AB2}).

Furthermore, the structure of Hopf algebra allows us to define new objects in the theory, named the comodule algebra and the partial comodule algebra, which are the duals of module algebra and partial module algebra, respectively (cf. \cite{CJ}). Alves and Batista have shown the relation between these structures through a globalization (cf.~\cite{AB3}).

Since coalgebra is the dual notion of algebra, thus one can think about these partial actions and coactions on a coalgebra, generating new objects and studying its properties. Batista and Vercruysse defined these structures, called partial module coalgebra and partial comodule coalgebra, which are the dual notions of partial module algebra and partial comodule algebra respectively. All these four partial structures are closely related. These dual structures were studied, and it was shown some interesting properties between them (cf.~\cite{BV}).

The aim of this work is to study the existence of globalization in the setting of partial module coalgebras and partial comodule coalgebras, relating our results with the results obtained by Alves and Batista in \cite{AB2}. This work is organized as follows.

In the second section we recall some preliminary results about partial actions of Hopf algebras, which are necessary for a full understanding of this work. An expert on this area can pass through this section, and go directly to the next section.

The third section is devoted to the study of partial actions on coalgebras. We study the correspondence between this new structure and the partial module algebra, some examples and useful properties. The purpose of this section is to discuss the existence of a globalization for partial action on coalgebras. For this, we start by introducing the notion of induced partial action on coalgebras and showing how to obtain it via a comultiplicative projection satisfying a special condition (see Proposition \ref{ipmc}). After that, we define a globalization for partial module coalgebras (see Definition \ref{gmc}) and we also discuss its affinity with the well known notion of globalization for partial module algebra, obtaining a direct relation between them (see Theorem \ref{gpmc<->gpma}). Finally, we construct a globalization for any given partial module coalgebra, called \emph{the standard globalization} (see Theorem \ref{teo-gmc}).

In the fourth section, we study partial coactions on coalgebras. We present some examples and important properties related to this structure. We also show a correspondence among %between
these four partial objects, asking for special conditions, like density or finite dimension, getting a one to one correspondence among %between
all of them. In a similar way as made in Section \ref{c2}, we define the induced partial coaction assuming  the existence of a comultiplicative projection satisfying a special condition (see Proposition \ref{ipcc}). After that, we define globalization for partial comodule coalgebra (see Definition \ref{gcc}) and relate it with the notion of partial module coalgebra earlier defined in Section \ref{c2} (see Theorem \ref{gcc+h0sp=>gmc}). Supposing a rational module hypothesis, we show that every partial comodule coalgebra is globalizable, constructing the \emph{standard} globalization for it (see Theorem \ref{teo-gcc}).

Throughout this paper $\Bbbk$ denotes a field, all objects are $\Bbbk$-vector spaces (i.e. algebra, coalgebra, Hopf algebra, etc., mean respectively $\Bbbk$-algebra, $\Bbbk$-coalgebra, $\Bbbk$-Hopf algebra, etc.), linear maps mean $\Bbbk$-linear maps and unadorned tensor product means $\otimes_\Bbbk$. We use the well known Sweedler's Notation for comodules and coalgebras, in this way: given a coalgebra $D$, with coproduct $\Delta$, we will denote, for any $c\in C$
\[
	\Delta (c) = c_1 \otimes c_2,
\]
where the summation is understood; and given a left $C$-comodule $M$ via $\lambda$, we will denote, for any $m\in M$
\[
	\lambda (m) = m^{-1} \otimes m^0,
\]
where the summation is also understood.

Moreover, we will call $\pi$ a \emph{projection} if it is a linear map such that $\pi \circ \pi = \pi$.

In the context of (partial) actions on algebras, it is usual to assume that all modules are left modules. In order to respect the correspondences among %between
the four partial structures (see Sections 3 and 4), we assume the following convention: (partial) module algebras are left (partial) module algebras; (partial) comodule algebras are right (partial) comodule algebras; (partial) module coalgebras are right (partial) module coalgebras; and (partial) comodule coalgebras are left (partial) comodule coalgebras.

%%%%%%%%%%%%%%%%%%%%%%%%%%%%%%%%%%%%%%%%%%%%%%%%%%%%%%%%%%%%%%%%%%%%%%%%%%%%%%%%%%%%%%%%%%%%
%%%%%%%%%%%%%%%%%%%%%                                %%%%%%%%%%%%%%%%%%%%%%%%%%%%%%%%%%%%%%%
%%%%%%%%%%%%%%%%%%%%%     MÓDULO ÁLGEBRA PARCIAL     %%%%%%%%%%%%%%%%%%%%%%%%%%%%%%%%%%%%%%%
%%%%%%%%%%%%%%%%%%%%%                                %%%%%%%%%%%%%%%%%%%%%%%%%%%%%%%%%%%%%%%
%%%%%%%%%%%%%%%%%%%%%%%%%%%%%%%%%%%%%%%%%%%%%%%%%%%%%%%%%%%%%%%%%%%%%%%%%%%%%%%%%%%%%%%%%%%%
\section{Partial Actions}

\begin{dfn}[Module algebra]\label{def-module-algebra}
	A \emph{left $H$-module algebra} is a pair $(A, \triangleright)$, where $A$ is an algebra and $\triangleright\colon H \otimes A \to A$ is a linear map, such that the following conditions hold, for any $h,k\in H$ and $a,b\in A$:
	\begin{enumerate}\Not{MA}
		\item $1_H \triangleright a = a$;\label{ma-1}
		\item $h \triangleright (a b) = (h_1 \triangleright a)(h_2 \triangleright b)$;\label{ma-2}
		\item $h \triangleright (k \triangleright a) = (hk) \triangleright a$;\label{ma-3}
	\end{enumerate}
\end{dfn}

\begin{obs}
	The standard definition of module algebra (cf. \cite{DNR}) contains the additional condition, for any $h\in H$:
	\[
		h \triangleright 1_A = \Ve(h) 1_A.
	\]
	This condition should be required for modules algebra over bialgebras. Since $H$ is a Hopf algebra, so it is a consequence of the others.
\end{obs}

\begin{dfn}[Partial module algebra \cite{CJ}]\label{pma}
	A \emph{left partial $H$-module algebra} is a pair $(A, \Pma)$, where $A$ is an algebra and $\Pma\colon H \otimes A \to A$ is a linear map, such that the following conditions are satisfied, for all $h,k\in H$ and $a,b\in A$:
	\begin{enumerate}\Not{PMA}
		\item $1_H \Pma a = a$;\label{pma-1}
		\item $h \Pma (ab) = (h_1 \Pma a)(h_2 \Pma b)$;\label{pma-2}
		\item $h \Pma (k \Pma a) = (h_1 \Pma 1_A)(h_2 k \Pma a)$.\label{pma-3}
	\end{enumerate}
	The partial action $\Pma$ is said \emph{symmetric} if the following additional condition is satisfied:
	\begin{enumerate}\Not[3]{PMA}
		\item $h \Pma (k \Pma a) = (h_1 k \Pma a)(h_2 \Pma 1_A)$.\label{pma-4}
	\end{enumerate}
\end{dfn}

The conditions $(\ref{pma-2})$ and $(\ref{pma-3})$ can be replaced by
\[
	h \Pma (a(g\Pma b)) = (h_1\Pma a)(h_2g\Pma b),
\]
for all $h, g \in H$ and $a,b \in A$. It can be useful if working with non-unital algebras.

A natural way to get a partial module algebra is inducing from a global action, as follows:
\def\ABIIPropI{\cite[Proposition~1]{AB2}}
\begin{prop}[\ABIIPropI]
	Given a (global) module algebra $B$ and a right ideal $A$ of $B$ with unity $1_A$. Then $A$ is a partial $H$-module algebra via
	\[
		h\cdot a = 1_A (h \triangleright a)
	\]
\end{prop}

In this context, an \emph{enveloping action} (or globalization) of a partial module algebra $A$ is a pair $(B, \theta)$, where $B$ is a module algebra, $\theta\colon A \to B$ is an algebra monomorphism, such that
\begin{enumerate}[\indent\bfseries (i)]
	\item $\theta(A)$ is a right ideal of $B$;
	\item the partial action on $A$ is equivalent to the induced partial action on $\theta(A)$;
	\item $B = H \triangleright\theta(A)$.
\end{enumerate}

Alves and Batista showed that any partial action has an enveloping action.
\def\ABIIThmI{\cite[Theorem~1]{AB2}}
\begin{teo}[\ABIIThmI]
	Let $A$ be a left partial $H$-module algebra, $\varphi\colon A \to \text{Hom}(H, A)$ the map given by $\varphi(a)(h) = h \cdot a$ and $B = H \triangleright \varphi(A)$. Then $(B, \varphi)$ is an enveloping action of $A$.
\end{teo}

\begin{obs}
	In the construction made above, we can highlight some points that are important for this paper.
	\begin{enumerate}[\indent\bfseries 1)]
		\item The image of $A$ should be a right ideal of an enveloping action $B$, but it does not need to be a left ideal, neither the partial action needs to be symmetric (cf. \cite[Proposition~4]{AB2}).
		\item The construction of an enveloping action supposes that the partial module algebra is unital. However, this restriction may be overcome by appropriate projections, as shown below.
		
		Let $B$ be an $H$-module algebra (not necessarily unital) and $A$ a subalgebra of $B$. Given a multiplicative projection
		\[
			\pi\colon B \longto A
		\]
		such that, for all $h,k \in H$ and $x,y \in A$, the condition
		\begin{align}
			\label{condindmap}
			\pi(h \Ma (x(k \Ma y))) = \pi(h \Ma (x\pi(k \Ma y)))
		\end{align}
		holds, so we can define a structure of partial module algebra  in $A$ by
		\begin{align*}
			h \cdot x = \pi(h \triangleright x).
		\end{align*}
		
		Note that the converse is also true. In fact, supposing that the projection $\pi$ induces a structure of partial module algebra in $A$, then Equation \eqref{condindmap} holds.
		
		Now, since we have the notion of induced partial action, we can define a globalization (or enveloping action) of $A$.
		
		\pagebreak
		\begin{dfn}[Globalization for partial module algebra]\label{gma}
			Given a right partial $H$-module algebra $A$ with partial action $\Pma$, a \emph{globalization} of $A$ is a triple  $(B, \theta, \pi)$, where $B$ is a right $H$-module algebra via $\Ma$, $\theta \colon A \to B$ is an algebra monomorphism and $\pi$ is a multiplicative projection from $B$ onto $\theta(A)$, satisfying the following conditions:
			\begin{enumerate}\NNot{GMA}
				\item the partial action on $A$ is equivalent to the partial action induced by $\Ma$ on $\theta(A)$, that is, $\theta(h \Pma a) = h \Pma \theta(a) = \pi(h \Ma \theta(a))$\label{gma-2}; 
				\item $B$ is the $H$-module algebra generated by $\theta(A)$, that is, $B = H \Ma \theta(A)$,
			\end{enumerate}
			for all $h \in H, a\in A$ and $b \in B$.
		\end{dfn} 
		
		It is a simple task to check that any partial module algebra has a globalization, in this sense.
		
		Since the induced partial action defined by Alves and Batista is a particular case of the construction above, where the projection is given by left multiplication by the idempotent $1_A$, then this construction of globalization generalizes the construction made in \cite{AB2}.
		
		\item Since the notion of induction by a central idempotent is a particular case of induction by projection, then it inspires us to define the induced partial (co)module coalgebra using projections (see Definitions~\ref{ipmc} and \ref{ipcc}).
	\end{enumerate}
\end{obs}

%%%%%%%%%%%%%%%%%%%%%%%%%%%%%%%%%%%%%%%%%%%%%%%%%%%%%%%%%%%%%%%%%%%%%%%%%%%%%%%%%%%%%%%%%%%%
%%%%%%%%%%%%%%%%%%%%%                                %%%%%%%%%%%%%%%%%%%%%%%%%%%%%%%%%%%%%%%
%%%%%%%%%%%%%%%%%%%%%   MÓDULO COÁLGEBRA PARCIAL     %%%%%%%%%%%%%%%%%%%%%%%%%%%%%%%%%%%%%%%
%%%%%%%%%%%%%%%%%%%%%                                %%%%%%%%%%%%%%%%%%%%%%%%%%%%%%%%%%%%%%%
%%%%%%%%%%%%%%%%%%%%%%%%%%%%%%%%%%%%%%%%%%%%%%%%%%%%%%%%%%%%%%%%%%%%%%%%%%%%%%%%%%%%%%%%%%%%
\section{Partial actions on coalgebras}\label{c2}

%%%%%%%%%%%%%%%%%%%%%%%%%%%%%%%%%%%%%%%%%%%%%%%%%%%%%%%%%%%%%%%%%%%%%%%%%%%%%%%%%%%%%%%%%%%%
%%%%%%%%%%%%%%%%%%%%%%%%%%%%%%%%%%%%%%%%%%%%%%%%%%%%%%%%%%%%%%%%%%%%%%%%%%%%%%%%%%%%%%%%%%%%
\subsection{Partial module coalgebras}\label{partial-module-coalgebra}
%%%%%%%%%%%%%%%%%%%%%%%%%%%%%%%%%%%%%%%%%%%%%%%%%%%%%%%%%%%%%%%%%%%%%%%%%%%%%%%%%%%%%%%%%%%%
%%%%%%%%%%%%%%%%%%%%%%%%%%%%%%%%%%%%%%%%%%%%%%%%%%%%%%%%%%%%%%%%%%%%%%%%%%%%%%%%%%%%%%%%%%%%

\begin{dfn}[Module coalgebra]\label{mc}
	A \emph{right $H$-module coalgebra} is a pair $(D,\Mc)$, where $D$ is a coalgebra and $\Mc\colon D \otimes H \to D$ is a linear map such that for any $g,h \in H$ and $d \in D$, the following properties hold:
	\begin{enumerate}\Not{MC}
		\item $d \Mc 1_H = d$;\label{mc-1}
		\item $\Delta( d \Mc  h) = d_1 \Mc h_1 \otimes d_2 \Mc h_2$;\label{mc-2}
		\item $(d \Mc h) \Mc g = d \Mc h g$.\label{mc-3}
	\end{enumerate}
	
	In this case we say that $H$ acts on $D$ via $\Mc$, or that $\Mc$ is an action of $H$ on $D$. We sometimes use the terminology global module algebra to differ the above structure from the partial one.
\end{dfn}

\begin{obs}
	The above definition can be seen in a categorical sense, as follows (cf.~\cite[Definition~11.2.8]{R}): \emph{a $\Bbbk$-vector space $D$ is said a right $H$-module coalgebra if it is a coalgebra object in the category of right $H$-modules}.
	
	Note that from this categorical approach, it is required a fourth property for an $H$-module coalgebra. More precisely, the following equality needs to hold:%MD
	\[
		\Ve_D(d \Mc h) = \Ve_D(d) \Ve_H(h),
	\]
	for all $h\in H$ and $d\in D$.
\end{obs}

Since $H$ is a Hopf algebra, so it has an antipode, then this additional condition is a consequence from the others axioms of Definition \ref{mc}, as stated below.

\begin{prop}\label{acmc}
	Let $H$ be a Hopf algebra and $D$ a right $H$-module coalgebra. Then $\Ve_D(d \Mc h) = \Ve_D(d) \Ve_H(h)$,
	for any $d\in D, h\in H$.\qed
\end{prop}

\begin{obs}
	From Definition \ref{mc} and Proposition \ref{acmc} it follows that one can define, without loss of generality, module coalgebra for non-counital coalgebras, and both definitions coincide when the coalgebra is counital.
\end{obs}

Now we present some classical examples of module coalgebras (cf. \cite{DNR,M,R}.

\begin{ex}\label{ex-mc-hopf}
	Any Hopf algebra is a right module coalgebra over itself via right multiplication.
\end{ex}

\begin{ex}\label{ex-mc-tensor}%ANOTACAO % Não iniciar um exemplo com "if".
	Let $D$ be a right $H$-module coalgebra and $C$ a coalgebra, then $C \otimes D$ is a right $H$-module coalgebra with action given by
	\[
		(c \otimes d) \Mc h = c \otimes (d \Mc h)
	\]
	for any $c \otimes d \in C \otimes D$ and $h\in H$.
\end{ex}

\begin{dfn}[Partial module coalgebra]\label{pmc}\cite[Definition~5.1.]{BV}
	A \emph{right partial $H$-module coalgebra} is a pair $(C, \Pmc)$, where $C$ is a coalgebra $C$ and $\Pmc\colon C \otimes H \to C$ is a linear map, such that the following conditions are satisfied, for any $g,h \in H$ and $c \in C$:
	\begin{enumerate}\Not{PMC}
		\item $c \Pmc 1_H = c $;\label{pmc-1}
		\item $\Delta( c \Pmc  h) = c_1 \Pmc h_1 \otimes c_2 \Pmc h_2$;\label{pmc-2}
		\item $(c \Pmc h) \Pmc g = \Ve(c_1 \Pmc h_1) (c_2 \Pmc h_2 g)$.\label{pmc-3}
	\end{enumerate}
	A partial module coalgebra is said \emph{symmetric} if the following additional condition is satisfied:
	\begin{enumerate}\Not[3]{PMC}
		\item $(c \Pmc h) \Pmc g = (c_1 \Pmc h_1 g) \Ve(c_2 \Pmc h_2)$.\label{pmc-4}
	\end{enumerate}
\end{dfn}

One can define left partial module coalgebra in an analogous way.

The definition of right partial module coalgebra can be extended to non-counital coalgebras, as follows.

\begin{dfn}\label{mcpse}
	A \emph{right partial $H$-module coalgebra} is a pair $(C,\Pmc)$, where $C$ is a (non necessarily counital) coalgebra and $\Pmc\colon C \otimes H \to C$ is a linear map, such that, for any $h,k \in H$ and $c \in C$, the following conditions hold:
	\begin{enumerate}\Not{PMC\textquotesingle}
		\item $c \Pmc 1_H = c$; \label{mcpse-1}
		\item $(c \Pmc h)_1 \otimes ((c \Pmc h)_2 \Pmc k) = (c_1 \Pmc h_1) \otimes (c_2 \Pmc h_2 k)$. \label{mcpse-2}
	\end{enumerate}	
	Moreover, it is \emph{symmetric} if the following additional condition holds:%MD
	\begin{enumerate}\Not[2]{PMC'}
		\item $((c \Pmc h)_1 \Pmc k) \otimes (c \Pmc h)_2 = (c_1 \Pmc h_1 k) \otimes (c_2 \Pmc h_2)$. \label{mcpse-3}
	\end{enumerate}
	
\end{dfn}

It is straightforward to check the following statement.

\begin{prop}
	If $C$ is a counital coalgebra, then Definition \ref{mcpse} is equivalent to Definition \ref{pmc}.\qed
\end{prop}

Batista and Vercruysse \cite{BV} related right partial module coalgebras and left partial module algebra, using non-degenerated dual pairing between an algebra $A$ and a coalgebra $C$. Here we consider a special case, where the algebra is the dual of the coalgebra.

\begin{prop}[{\cite[Theorem~5.12]{BV}}] \label{pma->pmc}
	Let $C$ be a coalgebra, and suppose that $\Pma\colon H\otimes C^\ast \to C^\ast$ and $\Pmc\colon C\otimes H \to C$ are linear maps satisfying the following compatibility:
	\begin{equation}
		(h \Pma \alpha)(c) = \alpha(c \Pmc h)\label{comp-pma<->pmc}
	\end{equation}
	for any $h\in H$ and $c\in C$.
	Then $C$ is a partial $H$-module coalgebra via $\Pmc$ if and only if $C^\ast$ is a partial $H$-module algebra via $\Pma$.
\end{prop}

\begin{obs}\label{Obs-Cqq}
	The existence of a linear map
	\[
		\begin{alignedat}1
			\Pmc\colon C\otimes H & \longto C
		\end{alignedat}
	\]
	implies the existence of a linear map
	\[
		\begin{alignedat}1
			\Pma\colon H \otimes C^\ast & \longto C^\ast \\
			h \otimes \alpha & \longmapsto (h \Pma \alpha)\colon c\longmapsto \alpha(c \Pmc h),
		\end{alignedat}
	\]
	for all $h\in H, \alpha\in C^\ast$ and $c\in C$. Clearly these two maps satisfy the Equation \eqref{comp-pma<->pmc}.
\end{obs}

\begin{obs}\label{Obs-reflex}
	It is not clear to the authors if the converse of Remark \ref{Obs-Cqq} is true in general. However, it holds if $C \simeq C^{\ast\ast}$ via
	\[
		\begin{alignedat}1
			\wedge\colon C & \longto C^{\ast\ast}\\
			c & \longmapsto \widehat{c}\colon\alpha\longmapsto\alpha(c).
		\end{alignedat}
	\]
	
	In fact, given $h\in H$ and $c\in C$ we consider $\xi_{h,c}\in C^{\ast\ast}$ given by
	\[
	\xi_{h,c}(\alpha) = (h\Pma \alpha)(c).
	\]
	
	Hence, we can define the linear map
	\[
		\begin{alignedat}1
			\Pmc\colon C \otimes H & \longto C\\
			c \otimes h & \longmapsto c \Pmc h = \wedge^{-1}(\xi_{h,c}),
		\end{alignedat}
	\]
	that clearly satisfies the Equation \eqref{comp-pma<->pmc}.
\end{obs}

The next result follows from Remarks \ref{Obs-Cqq} and \ref{Obs-reflex}.
\begin{prop}\label{pmc->pma}
	Let $C$ be a coalgebra. Then the following statements hold:
	\begin{enumerate}[(i)]
		\item If $C$ is a partial $H$-module coalgebra, then $C^{\ast}$ is a partial $H$-module algebra.
		\item If $C^\ast$ is a partial $H$-module algebra and $C \simeq C^{\ast\ast}$ via $\wedge$, then $C$ is a partial $H$-module coalgebra.\qed
	\end{enumerate}
\end{prop}

%%%%%%%%%%%%%%%%%%%%%%%%%%%%%%%%%%%%%%%%%%%%%%%%%%%%%%%%%%%%%%%%%%%%%%%%%%%%%%%%%%%%%%%%%%%%

%%%%%%%%%%%%%%%%%%%%%%%%%%%%%%%%%%%%%%%%%%%%%%%%%%%%%%%%%%%%%%%%%%%%%%%%%%%%%%%%%%%%%%%%%%%%

\subsubsection{Examples and the induced partial action}

\begin{ex}\label{expmc-1}
	Any global right $H$-module coalgebra is a partial one.
\end{ex}

\begin{ex}\label{expmc-2}
	Let $\alpha$ be a linear functional on $H$, then the ground field $\Bbbk$ is a right partial $H$-module coalgebra via $x \Pmc h = x\alpha(h)$ if and only if the following conditions hold, for any $x \in \Bbbk$ and $h, k \in H$:
	\begin{enumerate}[(i)]
		\item $\alpha(1_H) = 1_\Bbbk$;
		\item $\alpha(h)\alpha(k) = \alpha(h_1)\alpha(h_2 k)$.
	\end{enumerate}
\end{ex}

\begin{ex}
	Let $G$ be a group and $H = \Bbbk G$ the group algebra. Consider $\alpha \in \Bbbk G^\ast$ and let $N = \{g\in G\mid\alpha(g) \not= 0 \}$. Then $\Bbbk$ is a partial $\Bbbk G$-module coalgebra if and only if $N$ is a subgroup of $G$. In this case, we have that
	\[
		\alpha(g) =
		\begin{cases}
			1,&\text{if $g$ lies in $N$}\\
			0,&\text{otherwise.}
		\end{cases}
	\]
\end{ex}

\begin{ex}[{\cite[Theorem~5.7]{BV}}]
	A group $G$ acts partially on a coalgebra $C$ if and only if $C$ is a symmetric partial left $\Bbbk G$-module coalgebra. In this case, $g \cdot c = \theta_g(P_{g^{-1}}(c))$ and $P_g(c) = \Ve(g^{-1} \cdot c_1) c_2 = c_1 \Ve(g^{-1} \cdot c_2)$.
\end{ex}

Now, we construct a partial action on a coalgebra from a global one.

Let $D$ be a right $H$-module coalgebra and $C \subseteq D$ a subcoalgebra. Since $D$ is a right $H$-module, we could try to induce a partial action restricting the action of $D$ on $C$. But one can note that the range of this restriction does not need to be contained in $C$, thus we need to project it on $C$. By this way, let $\pi\colon D \to C$ be a projection from $D$ over $C$ (as vector spaces) and consider the following map
\[
	\Ipmc\colon C \otimes H \stackrel{\Mc}{\longrightarrow} D \stackrel{\pi}{\longrightarrow} C,
\]
where $\Mc$ denotes the right action of $H$ on $D$.

In the sequel, we exhibit a necessary and sufficient condition to the above map to be a partial action on $C$. In the next proposition, we denote by $C \Mc H$ the vector space spanned by the elements $c \Mc h$, for $c\in C$ and $h\in H$.

\begin{prop}[Induced partial module coalgebra]\label{ipmc}
	Let $D$ be a right $H$-module coalgebra, $C \subseteq D$ a subcoalgebra and $\pi \colon D \to C$ a comultiplicative projection satisfying
	\begin{equation}
		\pi [ \pi (x) \Mc h] = \pi [ \Ve(\pi (x_1)) x_2 \Mc h], \label{eq:pmc}
	\end{equation}
	for any $x \in C \Mc H$.
	
	Consider $\Ipmc\colon C\otimes H \to C$ the linear map given by
	\begin{equation}
		c \Ipmc h \coloneqq \pi(c \Mc h),\label{eq:ipmc}
	\end{equation}
	then $C$ becomes a right partial $H$-module coalgebra via $\Ipmc$.
\end{prop}
\begin{proof}\
	
	(\ref{pmc-1}): Let $c\in C$, so $c \Ipmc 1_H = \pi(c \Mc 1_H) = \pi(c) = c$, where the last equality holds because $\pi$ is a projection.
	
	(\ref{pmc-2}): If $\pi$ is a comultiplicative map (i.e., $\Delta \circ \pi = (\pi \otimes \pi) \circ \Delta$), then for any $c\in C$ and $h\in H$, we have
	\begin{alignat*}4
		\Delta(c \Ipmc h) & \EREf \Delta(\pi(c \Mc h))\\
			& \EREf (\pi \otimes \pi)(\Delta(c \Mc h))\\
			& \REf{mc-2}  (\pi \otimes \pi)(c_1 \Mc h_1 \otimes c_2 \Mc h_2)\\
			& \EREf \pi(c_1 \Mc h_1) \otimes \pi(c_2 \Mc h_2)\\
			& \EREf c_1 \Ipmc h_1 \otimes c_2 \Ipmc h_2.
	\end{alignat*}
	
	(\ref{pmc-3}): For $h,k \in H$ and $c\in C$, we have
	\begin{alignat*}4
		\Ve(c_1 \Ipmc h_1)(c_2 \Ipmc h_2 k)	& \EREf \Ve[\pi(c_1 \Mc h_1)]\pi(c_2 \Mc h_2 k)\\
			& \REf{mc-3}  \Ve[\pi(c_1 \Mc h_1)]\pi[(c_2 \Mc h_2) \Mc k] \\
			& \REf{mc-2}  \Ve[\pi((c \Mc h)_1)]\pi[((c \Mc h)_2) \Mc k] \\
			& \EREf \pi[\Ve[\pi((c \Mc h)_1)](c \Mc h)_2 \Mc k] \\
			& \REf{eq:pmc} \pi[\pi(c \Mc h) \Mc k] \\
			& \EREf (c \Ipmc h) \Ipmc k.
	\end{alignat*}
	
	Hence, $C$ is a partial module coalgebra.
\end{proof}

\begin{obs}
	One can note that the converse of Proposition \ref{ipmc} is also true. Indeed, supposing $C$ a subcoalgebra of a module coalgebra $D$ and $\pi \colon D \to C$ a comultiplicative projection such that $C$ is a partial module coalgebra by $c \Ipmc h = \pi(c \Mc h)$, hence $\pi [ \pi (x) \Mc h] = \pi [ \Ve(\pi (x_1)) x_2 \Mc h]$ for any $x \in C \Mc H$.
	
	The proof of the above statement follows straight from the calculations made in Proposition \ref{ipmc}.
\end{obs}

With the construction of induced partial action we have the necessary tools to define a globalization for a partial module coalgebra. This is our next goal.

%%%%%%%%%%%%%%%%%%%%%%%%%%%%%%%%%%%%%%%%%%%%%%%%%%%%%%%%%%%%%%%%%%%%%%%%%%%%%%%%%%%%%%%%%%%%
%%%%%%%%%%%%%%%%%%%%%%%%%%%%%%%%%%%%%%%%%%%%%%%%%%%%%%%%%%%%%%%%%%%%%%%%%%%%%%%%%%%%%%%%%%%%
\subsection{Globalization for partial module coalgebra}

From now on, given a left (resp., right) $H$-module $M$ with the action denoted by $\Ma$ (resp., $\Mc$), we consider the $\Bbbk$-vector space $H \Ma M$ (resp., $M \Mc H$) as the $\Bbbk$-vector space generated by the elements $h \Ma m$ (resp., $m \Mc h$) for all $h \in H$ and $m \in M$. Clearly, it is an $H$-submodule of $M$.

\begin{dfn}[Globalization for partial module coalgebra]\label{gmc}
	Given a right partial $H$-module coalgebra $C$ with partial action $\Pmc$, a \emph{globalization} of $C$ is a triple  $(D, \theta, \pi)$, where $D$ is a right $H$-module coalgebra via $\Mc$, $\theta \colon C \to D$ is a coalgebra monomorphism and $\pi$ is a comultiplicative projection from $D$ onto $\theta(C)$, satisfying the following conditions for all $h \in H, c\in C$ and $d \in D$:
	\begin{enumerate}\Not{GMC}\itemsep6pt
		\item $\pi[\pi(d) \Mc h] = \pi[\Ve(\pi(d_1))d_2 \Mc h]$;\label{gmc-1}
		\item $\theta(c \Pmc h) = \theta (c) \Ipmc h$;\label{gmc-2}
		\item $D$ is the $H$-module generated by $\theta(C)$, i.e., $D = \theta(C) \Mc H$.\label{gmc-3}
	\end{enumerate}
\end{dfn}

\begin{obs}
	The first condition of Definition \ref{gmc} says that we can induce a structure of partial module coalgebra on $\theta(C)$. The second one says that this induced partial action on $\theta(C)$ is equivalent to the partial action on $C$. The last one says that does not exists any submodule coalgebra of $D$ containing $\theta(C)$.
\end{obs}

%%%%%%%%%%%%%%%%%%%%%%%%%%%%%%%%%%%%%%%%%%%%%%%%%%%%%%%%%%%%%%%%%%%%%%%%%%%%%%%%%%%%%%%%%%%%
%%%%%%%%%%%%%%%%%%%%%%%%%%%%%%%%%%%%%%%%%%%%%%%%%%%%%%%%%%%%%%%%%%%%%%%%%%%%%%%%%%%%%%%%%%%%
\subsubsection{Correspondence between globalizations}

Our next aim is to establish relations between the globalization for partial module coalgebras, as defined in Definition \ref{gmc}, and for partial module algebras (cf.~\cite{AB2}). For this we use the fact that a partial module coalgebra $C$ naturally induces a structure of a partial module algebra on $C^\ast$.

Given a right partial $H$-module coalgebra $C$, it follows from Proposition \ref{pmc->pma} that the dual $C^\ast$ is a left partial $H$-module algebra. The same is true for (global) module coalgebras (cf.~\cite{R}).

\begin{obs}
	Let $C$ be a right partial $H$-module coalgebra, $D$ a right $H$-module coalgebra, $\theta\colon C \to D$ a coalgebra monomorphism and $\pi\colon D \to \theta(C)$ a comultiplicative projection. Since $\theta$ is injective, hence it has an inverse $\theta^{-1}$, defined in $\theta(C) = \pi(D)$.
	
	Consider the linear map $\varphi\colon C^{\ast} \to D^{\ast}$, given by the transpose of $\theta^{-1}\circ\pi$, i.e., for any $\alpha \in C^\ast$, we have $\varphi(\alpha) \coloneqq (\theta^{-1}\circ \pi)^\ast(\alpha) = \alpha \circ\theta^{-1}\circ \pi$. This is clearly a multiplicative monomorphism.
	
	So, one can also define the following linear maps:
	\begin{align*}
		\Ipmc \colon \pi(D) \otimes H & \longto \pi(D)\\
		\pi(d) \otimes h & \longmapsto \pi(\pi(d) \Mc h)\\
		\intertext{and}
		\Ipma \colon H \otimes \varphi(C^\ast) & \longto D^\ast \\
		h \otimes \varphi(\alpha) & \longmapsto \varphi(\varepsilon_C) \ast (h \Ma \varphi(\alpha)).
	\end{align*}
	
	Note that, there is a correspondence between these maps, given by
	\begin{eqnarray}\label{acoesinduzidas}
		(h \Ipma \varphi(\alpha))(\pi(d)) = \varphi(\alpha) (\pi(d) \Ipmc h),
	\end{eqnarray}
	for any $h \in H$, $d \in D$ and $\alpha \in C^\ast$. In fact,
	\begin{alignat*}{4}
		(h \Ipma \varphi(\alpha))(\pi(d))
			& \EREf (\varphi(\Ve)\ast (h\Ma\varphi(\alpha)))(\pi(d))\\
			& \EREf \varphi(\Ve)(\pi(d)_1)\ ((h\Ma\varphi(\alpha)))(\pi(d)_2)\\
			& \EREf \varphi(\Ve)(\pi(d)_1)\ \varphi(\alpha)(\pi(d)_2\Mc h)\\
			& \EREf \varphi(\Ve)(\pi(d_1))\ \varphi(\alpha)(\pi(d_2)\Mc h)\\
			& \EREf \Ve(\theta^{-1}(\pi(\pi(d_1))))\ \alpha(\theta^{-1}(\pi(\pi(d_2)\Mc h)))\\
			& \EREf \Ve_{\theta(C)}(\pi(d_1))\ \alpha(\theta^{-1}(\pi(\pi(\pi(d_2)\Mc h))))\\
			& \EREf \Ve_{\theta(C)}(\pi(d)_1)\ \varphi(\alpha)(\pi(\pi(d)_2\Mc h))\\
			& \EREf \varphi(\alpha)(\pi(\pi(d)\Mc h))\\
			& \EREf \varphi(\alpha)(\pi(d)\Ipmc h),
	\end{alignat*}
	for any $h \in H$, $d \in D$ and $\alpha \in C^\ast$.
	
	Notice that the maps $\Ipmc$ and $\Ipma$ are the induced partial actions on $\pi(D)$ and $\varphi(C^\ast)$, respectively.
	
	In fact, the next statement shows that the map $\Ipmc$ is the induced right partial action in $\pi(D)$ if and only if the map $\Ipma$ is the induced left partial action in $\varphi(C^\ast)$ and, moreover, it relates the globalization of the partial module coalgebra to	the globalization of the dual partial module algebra. Therefore, for simplicity we will write $\Pmc$ instead of $\Ipmc$ and $\Pma$ instead of $\Ipma$ (even for induced partial actions).
\end{obs}

\begin{teo}\label{gpmc<->gpma}
	Let $C$ be a partial module coalgebra. With the above notations, we have that $(\theta(C) \Mc H, \theta, \pi)$ is a globalization for $C$ if and only if $(H \Ma \varphi(C^\ast), \varphi)$ is a globalization for $C^\ast$.
\end{teo}
\begin{proof}
	If $(\theta(C) \Mc H, \theta, \pi)$ is a globalization for $C$, it follows that
	\begin{alignat*}4
		(\varphi(\Ve_C) \ast (h \Ma \varphi(\alpha)))(d)
			& \EREf (\varphi(\Ve_C)(d_1))((h \Ma \varphi(\alpha))(d_2)) \\
			& \EREf \Ve_C(\theta^{-1} (\pi (d_1)))  \varphi(\alpha)(d_2 \Mc h) \\
			& \EREf \Ve_{\theta(C)}(\pi (d_1))  \alpha(\theta^{-1} (\pi (d_2 \Mc h))) \\
			& \EREf \alpha(\theta^{-1} (\pi (\Ve_{\theta(C)}(\pi (d_1))d_2 \Mc h))) \\
			& \REf{gmc-1}  \alpha(\theta^{-1} (\pi (\pi(d) \Mc h))) \\
			& \REf{eq:ipmc}  \alpha(\theta^{-1} (\pi(d) \Pmc h)) \\
			& \REf{gmc-2}  \alpha((\theta^{-1} (\pi(d))) \Pmc h) \\
			& \REf{pmc->pma}  ( h \Pma \alpha)(\theta^{-1} (\pi(d))) \\
			& \EREf \varphi(h \Pma \alpha)(d),
	\end{alignat*}
	for every $h\in H, \alpha \in C^{\ast}$ and $d\in D$. Thus,  $\varphi(C^\ast)$ is a right ideal of $H \Ma \varphi(C^\ast)$ and, moreover, $h \Pma \varphi(\alpha) = \varphi(h \Pma \alpha)$. Therefore, $(H \Ma \varphi(C^\ast), \varphi)$ is a globalization for $C^\ast$, as desired.
	
	Conversely, if $(H \Ma \varphi(C^\ast), \varphi)$ is a globalization for $C^\ast$, then
	
	(\ref{gmc-1}): Given $\alpha \in C^{\ast}, h\in H$ and $d\in D$ we have
	\begin{alignat*}4
		\alpha(\theta^{-1} (\pi(\pi(d) \Mc h)))
			& \EREf \alpha(\theta^{-1}( \pi(\pi(d) \Pmc h)))\\
			& \EREf \varphi(\alpha)(\pi(d) \Pmc h)\\
			& \REf{acoesinduzidas} (h \Pma \varphi(\alpha))(\pi(d))\\
			& \REf{gma-2}  \varphi(h \Pma \alpha)(\pi(d))\\
			& \EREf (h \Pma \alpha)(\theta^{-1}(\pi(\pi(d))))\\
			& \EREf (h \Pma \alpha)(\theta^{-1}(\pi(d)))\\
			& \EREf \varphi(h \Pma \alpha)(d) \\
			& \EREf (\varphi(\Ve_C)\ast(h \Ma \varphi(\alpha)))(d) \\
			& \EREf \varphi(\Ve_C)(d_1)(h \Ma \varphi(\alpha))(d_2) \\
			& \EREf \varphi(\Ve_C)(d_1) \varphi(\alpha)(d_2 \Mc h) \\
			& \EREf \varphi(\Ve_C)(d_1) \alpha(\theta^{-1} (\pi(d_2 \Mc h))) \\
			& \EREf \Ve_{C}(\theta^{-1} (\pi(d_1))) \alpha(\theta^{-1} (\pi (d_2 \Mc h))) \\
			& \EREf \Ve(\pi(d_1)) \alpha(\theta^{-1} (\pi (d_2 \Mc h))) \\
			& \EREf \alpha( \theta^{-1} (\pi (\Ve(\pi(d_1))d_2 \Mc h))).
	\end{alignat*}
	Since $\alpha$ is an arbitrary linear functional on $C$ and $\theta \circ \theta^{-1} = I_C$, then
	\[
		\pi(\pi(d) \Mc h) = \pi (\Ve(\pi(d_1))d_2 \Mc h).
	\]
	
	(\ref{gmc-2}): Given $\alpha\in C^\ast$, $h\in H$ and $c\in C$, then
	\begin{alignat*}4
		(\alpha \circ \Ti)(\theta(c) \Pmc h)
			& \EREf (\alpha \circ \Ti) (\pi (\theta(c) \Mc h))\\
			& \EREf (\alpha \circ \Ti \circ \pi) (\theta(c) \Mc h)\\
			& \EREf \varphi(\alpha) (\theta(c) \Mc h)\\
			& \EREf [h \Ma \varphi(\alpha)] (\theta(c))\\
			& \EREf [h \Ma \varphi(\alpha)] (\theta(\Ve_{C}(c_1)~ c_2))\\
			& \EREf \Ve_{C}(c_1)~ [h \Ma \varphi(\alpha)] (\theta(c_2)) \\
			& \EREf (\Ve_{C} \circ \Ti \circ \theta)(c_1)~ [h \Ma \varphi(\alpha)] (\theta(c_2)) \\
			& \EREf (\Ve_{C} \circ \Ti \circ \pi \circ \theta)(c_1)~ [h \Ma \varphi(\alpha)] (\theta(c_2)) \\
			& \EREf \varphi (\Ve_C)(\theta(c_1)) ~ [h \Ma \varphi(\alpha)] (\theta(c_2)) \\
			& \EREf \varphi (\Ve_C)(\theta(c)_1)~ [h \Ma \varphi(\alpha)] (\theta(c)_2) \\
			& \EREf [\varphi (\Ve_C)\ast(h \Ma \varphi(\alpha))]~(\theta(c))\\
			& \EREf (h \Pma \varphi(\alpha))~ (\theta(c))\\
			& \REf{gma-2} \varphi(h \Pma \alpha)~ (\theta(c))\\
			& \EREf (h \Pma \alpha)~ ( \Ti (\pi (\theta(c))))\\
			& \EREf (h \Pma \alpha)~ (\Ti (\theta(c)))\\
			& \EREf (h \Pma \alpha) (c)\\
			& \REf{pmc->pma} \alpha~ (c \Pmc h)\\
			& \EREf (\alpha \circ \Ti \circ \theta) (c \Pmc h)\\
			& \EREf (\alpha\circ \Ti)(\theta (c \Pmc h)).
	\end{alignat*}
	
	Since $\alpha\in C^\ast$ is arbitrary, we obtain that $\theta(c) \Pmc h = \theta (c \Pmc h)$. Therefore, $(\theta(C) \Mc H, \theta, \pi)$ is a globalization for $C$.
\end{proof}

%%%%%%%%%%%%%%%%%%%%%%%%%%%%%%%%%%%%%%%%%%%%%%%%%%%%%%%%%%%%%%%%%%%%%%%%%%%%%%%%%%%%%%%%%%%%
%%%%%%%%%%%%%%%%%%%%%%%%%%%%%%%%%%%%%%%%%%%%%%%%%%%%%%%%%%%%%%%%%%%%%%%%%%%%%%%%%%%%%%%%%%%%
\subsubsection{The standard globalization}

Now our next aim is to show that every partial module coalgebra has a globalization, constructing the standard globalization.

\begin{obs}
	Let $C$ be a right partial $H$-module coalgebra. Consider the coalgebra $C \otimes H$ with the natural structure of the tensor coalgebra, the coalgebra monomorphism from $C$ into $C \otimes H$, given by the natural embedding
	\[
		\begin{alignedat}1
			\varphi\colon C & \longto  C\otimes H\\
			c & \longmapsto  c \otimes 1_H,
		\end{alignedat}
	\]
	and the projection from $C \otimes H$ onto $\varphi(C)$, given by
	\[
		\begin{alignedat}1
			\pi\colon C\otimes H & \longto  \theta(C)\\
			c\otimes h & \longmapsto  (c \Pmc h) \otimes  1_H.
		\end{alignedat}
	\]
	We claim that $\pi$ is comultiplicative. Indeed,  for  $c\in C$ and $h\in H$ we have
	\begin{alignat*}4
		\Delta (\pi(c \otimes h))
			& \EREf \Delta ((c \Pmc h)\otimes 1_H)\\
			& \EREf (c \Pmc h)_1\otimes 1_H \otimes (c \Pmc h)_2\otimes 1_H\\
			& \REf{pmc-2}  c_1 \Pmc h_1\otimes 1_H \otimes c_2 \Pmc h_2\otimes 1_H\\
			& \EREf \pi(c_1 \otimes h_1) \otimes \pi(c_2 \otimes h_2)\\
			& \EREf (\pi \otimes \pi) \Delta(c \otimes h).
	\end{alignat*}
\end{obs}

With the above noticed we are able to construct a globalization for a partial module coalgebra.

\begin{teo}\label{teo-gmc}
	Every right partial $H$-module coalgebra has a globalization.
\end{teo}
\begin{proof}
	Let $C$ be a left partial $H$-module coalgebra, so from Examples \ref{ex-mc-hopf} and \ref{ex-mc-tensor}, we know that $C \otimes H$ is an $H$-module coalgebra, with action given by right multiplication in $H$. By the above noticed, we have the maps ${\varphi\colon C \to C \otimes H}$ and ${\pi\colon C \otimes H \to \varphi(C)}$, as required in Definition \ref{gmc}. Then we only need to show that the conditions (\ref{gmc-1}) and (\ref{gmc-2}) hold.
	
	(\ref{gmc-1}): For every $h, k \in H$ and $c\in C$, we have
	\begin{alignat*}4
		\pi[\Ve(\pi((c\otimes h)_1))(c\otimes h)_2 \Mc k]
			& \REf{pmc-2}  \pi[\Ve(\pi(c_1\otimes h_1))(c_2\otimes h_2) \Mc k]\\
			& \EREf \Ve(c_1\Pmc h_1)\pi[(c_2\otimes h_2) \Mc k]\\
			& \EREf \Ve(c_1\Pmc h_1)\pi[c_2\otimes h_2 k]\\
			& \EREf \Ve(c_1\Pmc h_1)(c_2\Pmc h_2 k)\otimes 1_H\\
			& \REf{pmc-3}  (c\Pmc h) \Pmc k \otimes 1_H\\
			& \EREf \pi[(c\Pmc h) \otimes k]\\
			& \EREf \pi[((c\Pmc h) \otimes 1_H) \Mc k]\\
			& \EREf \pi[\pi(c\otimes h) \Mc k].
	\end{alignat*}
	
	(\ref{gmc-2}): Let $h\in H$ and $c\in C$, then
	\begin{alignat*}4
		\varphi (c) \Pmc h
			& \EREf \pi[\varphi (c) \Mc h]\\
			& \EREf \pi[(c \otimes 1_H) \Mc h]\\
			& \EREf \pi[c \otimes  h]\\
			& \EREf c \Pmc  h \otimes 1_H\\ 
			& \EREf \varphi(c \Pmc  h). 
	\end{alignat*}
	
	Moreover, by the definitions of $\pi, \varphi$ and $\Mc$ it follows that $\varphi(C) \Mc H = C \otimes H$. Therefore $C\otimes H$ is a globalization for $C$.
\end{proof}

The globalization above constructed is called the \emph{standard globalization} and it is close related with the \emph{standard globalization} for partial module algebras, as follows.

\begin{teo}\label{gpmc<->gpma-standard}
	Let $C$ be a right partial $H$-module coalgebra, then the standard globalization for $C$ generates the standard globalization for $C^\ast$ as left partial $H$-module algebra.
\end{teo}

\begin{proof}
	From Theorem \ref{teo-gmc}, we have that $({C \otimes H}, \varphi, \pi)$ is the standard globalization for $C$. Consider the multiplicative map
	\[
		\begin{alignedat}1
			\phi\colon C^\ast & \longto  (C \otimes H)^\ast\\
			\alpha & \longmapsto  \alpha \circ \varphi^{-1}\circ\pi.
		\end{alignedat}
	\]
	Thus, by the Theorem \ref{gpmc<->gpma}, we have that $(H \Ma \phi(C^\ast),\phi)$ is a globalization for $C^\ast$, where the action on $(C \otimes H)^\ast$ is given by
	
	\[
		\begin{alignedat}1
			\Ma\colon H \otimes (C \otimes H)^\ast & \longrightarrow (C \otimes H)^\ast\\
			h \otimes \xi & \longmapsto  (h \Ma \xi)(c \otimes k) = \xi (c \otimes k\, h),
		\end{alignedat}
	\]
	for every $\xi\in (C \otimes H)^\ast$, $c \in C$ and $h,k \in H$.
	
	Now, consider the following algebra isomorphism given by the adjoint isomorphism
	\[
		\begin{alignedat}1
			\Psi\colon(C \otimes H)^\ast & \longto  \Hom(H,\,C^\ast)\\
			\xi & \longmapsto  [\Psi(\xi) (h)](c) = \xi(c \otimes h),
		\end{alignedat}
	\]
	which is an $H$-module morphism. In fact, let $h,\,k\in H$, $c\in C$ and $\xi\in(C\otimes H)^\ast$, so
	\begin{alignat*}4
		\{[\Psi(h \Ma \xi)](k)\}(c)
			& \EREf [(h \Ma \xi)](c\otimes k)\\
			& \EREf \xi(c\otimes k\,h)\\
			& \EREf \{[\Psi(\xi)](k\,h)\}(c)\\
			& \EREf \{[h\Ma\Psi(\xi)](k)\}(c)
	\end{alignat*}
	and, therefore, $\Psi$ is an $H$-module map. Moreover, composing $\Psi$ with $\phi$ we obtain
	\begin{alignat*}4
		\{[\Psi\circ\phi(\alpha)](h)\}(c)
			& \EREf \phi(\alpha)(c \otimes h)\\
			& \EREf \alpha(\varphi^{-1}(\pi(c \otimes h)))\\
			& \EREf \alpha(\varphi^{-1}(c \Pmc h\otimes 1_H))\\
			& \EREf \alpha(c \Pmc h)\\
			& \EREf (h \Pma \alpha)(c)\\
			& \EREf [\Phi(\alpha)(h)](c),
	\end{alignat*}
	where $\Phi\colon C^\ast \to \Hom(H,\,C^\ast)$, given by $\Phi(\alpha)(h) = h \Pma \alpha$, for all $h\in H$ and $\alpha\in C^\ast$ is the multiplicative map that appears in the construction of the standard globalization, replacing $A$ by $C^\ast$ (cf.~\cite[Theorem~1]{AB2}).
\end{proof}

%%%%%%%%%%%%%%%%%%%%%%%%%%%%%%%%%%%%%%%%%%%%%%%%%%%%%%%%%%%%%%%%%%%%%%%%%%%%%%%%%%%%%%%%%%%%
%%%%%%%%%%%%%%%%%%%%%                                %%%%%%%%%%%%%%%%%%%%%%%%%%%%%%%%%%%%%%%
%%%%%%%%%%%%%%%%%%%%%   COMÚDULO COÁLGEBRA PARCIAL   %%%%%%%%%%%%%%%%%%%%%%%%%%%%%%%%%%%%%%%
%%%%%%%%%%%%%%%%%%%%%                                %%%%%%%%%%%%%%%%%%%%%%%%%%%%%%%%%%%%%%%
%%%%%%%%%%%%%%%%%%%%%%%%%%%%%%%%%%%%%%%%%%%%%%%%%%%%%%%%%%%%%%%%%%%%%%%%%%%%%%%%%%%%%%%%%%%%
\section{Partial coaction on coalgebras} \label{c3}

%%%%%%%%%%%%%%%%%%%%%%%%%%%%%%%%%%%%%%%%%%%%%%%%%%%%%%%%%%%%%%%%%%%%%%%%%%%%%%%%%%%%%%%%%%%%
%%%%%%%%%%%%%%%%%%%%%%%%%%%%%%%%%%%%%%%%%%%%%%%%%%%%%%%%%%%%%%%%%%%%%%%%%%%%%%%%%%%%%%%%%%%%
\subsection{Partial comodule coalgebra}

Given two vector spaces  $V$ and $W$, we write $\tau_{\scriptscriptstyle V,W}$ to denote the standard isomorphism between $V \otimes W$ and $W \otimes V$.

\begin{dfn}[Comodule coalgebra]\label{cc}
	A \emph{left $H$-comodule coalgebra} is a pair $(D, \lambda)$, where $D$ is a coalgebra and $\lambda\colon D \to H \otimes D$ is a linear map, such that, for all $d\in D$, the following conditions hold:
	\begin{enumerate}\Not{CC}
		\item $(\Ve_H \otimes I)\lambda(d) = d$;\label{cc-1}
		\item $(I \otimes \Delta_D)\lambda (d) = (m_H \otimes I \otimes I ) (I \otimes \tau_{D, H} \otimes I)(\lambda \otimes \lambda)\Delta_D (d)$;\label{cc-2}
		\item $(I \otimes \lambda)\lambda (d) = (\Delta_H \otimes I)\lambda(d) $\label{cc-3}.
	\end{enumerate}
	In this case, we say that \emph{$H$ coacts on} $D$ via $\lambda$, or that $\lambda$ is a \emph{coaction of $H$ on $D$}. We will also call it a global comodule coalgebra to differ explicitly from the partial one.
\end{dfn}

We can also see the above definition in a categorical approach, in the following sense (cf.~\cite[Definition~11.3.7]{R}): \emph{a $\Bbbk$-vector space $D$ is a left $H$-comodule coalgebra if it is a coalgebra object in the category of left $H$-comodules.}

From this categorical point of view, one additional condition is required in Definition \ref{cc}, that is,
\begin{equation}\label{cond-add-cc}
	(I \otimes \Ve_D)\lambda(d) = \Ve_D(d)1_H,
\end{equation}
for all $d\in D$.

Since $H$ is a Hopf algebra (so it has an antipode), thus this additional condition may be obtained from the another ones, as stated below.

\begin{prop}\label{accc}
	Let $D$ be a left $H$-comodule coalgebra in the sense of Definition \ref{cc}. Then $(I \otimes \Ve_D)\lambda(d) = \Ve_D(d)1_H,$ for any $d \in D$.\qed
\end{prop}

\begin{obs}
	From the above proposition one can extend the notion of comodule coalgebra for non-counital coalgebras.
\end{obs}

Now we exhibit some classical examples of comodule coalgebras (cf. \cite[Section~11.3]{R}).

\begin{ex}
	A Hopf algebra $H$ is an $H$-comodule coalgebra with coaction ${\lambda\colon H \to H \otimes H}$ given by
	\[
		\lambda(h) = h_1 S(h_3) \otimes  h_2,
	\]
	for any $h\in H$.
\end{ex}

\begin{ex}
	Any coalgebra $D$ is an $H$-comodule coalgebra with coaction ${\lambda\colon D \to H \otimes D}$ defined by $\lambda(d) = 1_H \otimes d$, for every $d\in D$.
\end{ex}

\begin{ex}
	If the Hopf algebra $H$ is finite dimensional, then $H^\ast$ is an $H$-comodule coalgebra with structure given by $\lambda\colon H^\ast \to H \otimes H^\ast$ with $\lambda(f) = \sum\limits_{i=1}^{n} h_i \otimes f\ast h_i^\ast$, where $\{h_i\}_{i=1}^n$ and $\{h_i^\ast\}_{i=1}^n$ are dual basis for $H$ and $H^\ast$, respectively.
\end{ex}

\begin{ex}
	Let $C$ be a left $H$-comodule coalgebra with coaction $\lambda$ and $D$ a coalgebra, then $C \otimes D$ is a left $H$-comodule coalgebra via $\lambda \otimes I_D$.
\end{ex}

%%%%%%%%%%%%%%%%%%%%%%%%%%%%%%%%%%%%%%%%%%%%%%%%%%%%%%%%%%%%%%%%%%%%%%%%%%%%%%%%%%%%%%%%%%%%
%%%%%%%%%%%%%%%%%%%%%%%%%%%%%%%%%%%%%%%%%%%%%%%%%%%%%%%%%%%%%%%%%%%%%%%%%%%%%%%%%%%%%%%%%%%%
\subsubsection{Definitions and correspondences}

\begin{dfn}[Partial comodule coalgebra]\label{pcc}{\rm\cite[Definition~6.1]{BV}}
	A \emph{left partial $H$-comodule coalgebra} is a pair $(C,\lambda')$, where $C$ is a coalgebra and $\lambda' \colon C \to H \otimes C$ is a linear map, such that, for any $c\in C$, the following conditions hold:
	\begin{enumerate}\Not{PCC}
		\item $(\Ve_H \otimes I)\lambda'(c) = c$;\label{pcc-1}
		\item $(I \otimes \Delta_C)\lambda' (c) = (m_H \otimes I \otimes I ) (I \otimes \tau_{C,H} \otimes I)(\lambda' \otimes \lambda')\Delta_C (c)$;\label{pcc-2}
		\item $(I \otimes \lambda')\lambda' (c) = (m_H \otimes I \otimes I)\{ \nabla \otimes [(\Delta_H \otimes I)\lambda']\}\Delta_C (c)$,\label{pcc-3}
	\end{enumerate}
	where $\nabla\colon C \to H$ is defined by $\nabla(c) = (I \otimes \Ve_C)\lambda'(c)$.%, for all $c\in C$.
	
	The partial comodule coalgebra is said \emph{symmetric} if the following additional condition holds, for any $c\in C$:
	\begin{enumerate}\Not[3]{PCC}
		\item $(I \otimes \lambda')\lambda' (c) = (m_H \otimes I \otimes I)(I \otimes \tau_{{}_{H\otimes C, H}})\{[(\Delta_H \otimes I)\lambda'] \otimes \nabla\}\Delta_C (c)$.\label{pcc-4}
	\end{enumerate}
\end{dfn}

\begin{obs}
	For a partial comodule coalgebra $C$ via $\lambda'$ we use the Sweedler's notation $\lambda'(c) = c\Up \otimes c\Zp$, where the summation is understood. The bar over the upper index is useful to distinguish partial from global comodule coalgebras when working with both structures in a single computation.
\end{obs}

\begin{prop}{\rm \cite[Lemma~6.3~and~Corollary~6.4]{BV}}
	Let $C$ be left partial $H$-comodule coalgebra, then, for all $c \in C$, the following equa\-li\-ties hold:
	\begin{equation}
		c\Up \otimes c\Zp = \nabla(c_1) c_2\Up \otimes c_2 \Zp = c_1\Up \nabla(c_2) \otimes c_1\Zp \label{nabla-pcc2}
	\end{equation}
	and
	\begin{equation}
		\nabla(c_1) \nabla(c_2) = \nabla(c). \label{nabla-idem}
	\end{equation}
\end{prop}

\begin{ex}
	Every left $H$-comodule coalgebra is a left partial $H$-comodule coalgebra.
\end{ex}

The next result gives us a simple method to construct new examples of partial comodule coalgebras. The proof is straightforward and it will be omitted.

\begin{prop}
	Let $\lambda'\colon \Bbbk \to H \otimes \Bbbk$ be a linear map and $h\in H$ such that $\lambda'(1_\Bbbk) = h \otimes 1_\Bbbk$. Then the ground field $\Bbbk$ is a left partial $H$-comodule coalgebra if and only if the following conditions hold:
	\begin{enumerate}
		\item $\Ve_H(h) = 1_\Bbbk$;
		\item $h \otimes h = (h \otimes 1_H)\Delta(h)$.\qed
	\end{enumerate}
\end{prop}

As an application of the above result, we present the next example.

\begin{ex}
	Let G be a group, $\lambda'\colon \Bbbk \to \Bbbk G \otimes \Bbbk$ a linear map and $x = \sum\limits_{g \in G} \alpha_g g$ in $\Bbbk G$ such that $\lambda'(1)=x \mathop{\otimes} 1$.  Consider $N=\{g \in G \mid \alpha_g \neq 0 \}$ and suppose that the characteristic of $\Bbbk$ does not divides $|N|$. Then $\Bbbk$ is a left partial $\Bbbk G$-comodule coalgebra via $\lambda'$ if and only if $N$ is a finite subgroup of $G$. In this case, we have that
	\[
		\alpha_g = \dfrac{1}{|N|},
	\]
	for all $g\in N$
\end{ex}

\begin{prop}{\cite{BV}}\label{pcc:cond-global}
	Let $C$ be a left partial $H$-comodule coalgebra via $\lambda'$, then it is a (global) $H$-comodule coalgebra if and only if 
	\[
		c\Up \Ve_C(c\Zp) = \Ve_C (c) 1_H,
	\]
	for all $c\in C$.\hfill\qedsymbol
\end{prop}

Given a Hopf algebra $H$, we say that the finite dual $H^0$ \emph{separate points} if it is dense on $H^\ast$ in the finite topology, i.e., if $h\in H$ is such that $f(h)=0$, for all $f\in H^0$, then $h=0$.

For a coalgebra $C$ and a linear map $\lambda'\colon C \to H \otimes C$ (denoting by $\lambda' (c) = c\Up \otimes c\Zp$) we have two induced linear maps $\Pmci\colon C \otimes H^\ast \to C$ and $\Pmai\colon H^\ast \otimes C^\ast \to C^\ast$, given respectively by
\begin{align}
	c \Pmci f & =  f(c\Up) c\Zp\\
	(f \Pmai \alpha)(c) & =  f(c\Up) \alpha(c\Zp),
\end{align}
for all $c\in C, \alpha \in C^{\ast}$ and $f\in H^{\ast}$. Since, in general, $H^\ast$ is not a Hopf algebra, thus we can restrict $\Pmci$ and $\Pmai$ to the subspaces $C \otimes H^0$ and $H^0 \otimes C^\ast$, respectively. Therefore, under the assumption that $C$ is a left partial $H$-comodule coalgebra, we can show the following.

\begin{teo}\label{inducedfromcomcoalg}
	With the above notations, if $C$ is a left partial $H$-comodule coalgebra via $\lambda'$, then the following statements hold:
	\begin{enumerate}\Not{}
		\item $C$ is a right partial $H^0$-module coalgebra via $\Pmci$; \label{inducedfromcomcoalg-1}
		\item $C^\ast$ is a left partial $H^0$-module algebra via $\Pmai$. \label{inducedfromcomcoalg-2}
	\end{enumerate}
\end{teo}
\begin{proof}
	
	$(\ref{inducedfromcomcoalg-1}):$ Taking the dual pairing between $H$ and $H^0$ given by evaluation, then by Theorem~6.7 of \cite{BV} we have the desired.
	
	$(\ref{inducedfromcomcoalg-2}):$ In this case, taking the dual pairing between $C$ and $C^\ast$ given by evaluation, then the desired follows from Theorem~6.8 of \cite{BV}.
\end{proof}

It is not clear if the converse of the above theorem is true in general, but whenever $H^0$ separate points (so the dual pairing between $H$ and $H^0$ is non-degenerate) it holds, as stated in the next theorem.

\begin{teo}\label{h0seppontops-pcc-pmci-pmai}
	With the above notations, if $H^0$ separate points, then the following conditions are equivalent:
	\begin{enumerate}\Not{}
		\item $C$ is a right partial $H^{0}$-module coalgebra via $\Pmci$;\label{pmcih0sepponntos}
		\item $C$ is a left partial $H^{0}$-module algebra via $\Pmai$;\label{pmaih0sepponntos}
		\item $C$ is a left partial $H$-comodule coalgebra via $\lambda'$.\label{pcch0sepponntos}\hfill\qedsymbol
	\end{enumerate}
\end{teo}

The above theorem can be translated in the following commutative diagram:

\begin{equation}\label{diag-pcc->mpc+pma}
	\begin{gathered}
		\xymatrix{
			(C, \lambda', H)\ar@<-4pt>[rrrr]\ar@<4pt>@{<--}[rrrr]^-{\mathrm{H^0\ sep\ points}}\ar@<4pt>[dd]\ar@<-4pt>@{<--}[dd]_-{\rm H^0\ sep\ points} & & & &(C^\ast, \Pmai, H^0)\\
			\\
			(C, \Pmci, H^0)\ar@{<->}[uurrrr]
		}
	\end{gathered}
\end{equation}

Theorems \ref{inducedfromcomcoalg} and \ref{h0seppontops-pcc-pmci-pmai} show relations between partial comodule coalgebra, partial module coalgebra and partial module algebra, whenever we start from a partial coaction $\lambda'$. In general, we can not start from an action and induce a coaction. To do this we require a more strong hypothesis on $H$, more precisely, we assume that $H$ is finite dimensional.

In fact, if $H$ is finite dimensional, then $H^0 = H^\ast$ (and so $H^0$ separate points). Moreover, given a linear map $\Pmc\colon C \otimes H^\ast \to C$  and a dual basis $\{h_i, h_i^\ast\}_{i=1}^n$ for $H$ and $H^\ast$, then one can induce a linear map $\Pcci\colon C \to H \otimes C$, by
\[
	\Pcci(c) = \sum\limits_{i=1}^n h_i \otimes c \Pmc h_i^\ast
\]
and it is clear that
\[	
	f(c\Up)c\Zp = c \Pmc f,
\]
for all $c \in C$ and $f \in H^\ast$.

Thus, for a finite dimensional Hopf algebra $H$, we can induce a coaction of $H$ on a coalgebra $C$ from a given action of $H^{\ast}$ on $C$.

In \cite{AB1}, there is a similar construction for (right) partial $H$-comodule algebras and (left) partial $H^\ast$-module algebras. Hence we have that these four partial structures are close related, in the following sense:

\begin{teo}\label{p(co)m(co)a-equiv}
	Let $C$ be a coalgebra and suppose that the Hopf algebra $H$ is finite dimensional. Then the following statements are equivalent:
	\begin{enumerate}\Not{}
		\item $C$ is a left partial $H$-comodule coalgebra;
		\item $C^\ast$ is a right partial $H$-comodule algebra;
		\item $C$ is a right partial $H^\ast$-module coalgebra;
		\item $C^\ast$ is a left partial $H^\ast$-module algebra.
	\end{enumerate}
	The relations between the correspondent actions and coactions are given in the following way, for any $c\in C,\ \alpha\in C^\ast$ and $f\in H^\ast$:
	\begin{align}
		\alpha (c\Up) c\Zp	& =  \alpha^{\overline{0}}(c) \alpha^{+\overline{1}}\label{pcc<=>pca} \\
		(f \Pma \alpha)(c)	& =  \alpha(c \Pmc f)\\
		c \Pmc f			& =  f(c\Up) c\Zp \\
		f \Pma \alpha		& =  \alpha^{\overline{0}} f(\alpha^{+\overline{1}}),
	\end{align}
	where $\lambda'\colon c \mapsto c\Up \otimes c\Zp$ and $\rho'\colon\alpha \mapsto \alpha^{\overline{0}} \otimes \alpha^{+\overline{1}}$ are the partial coactions on $C$ and $C^\ast$, respectively.\qed
\end{teo}

Using Theorem \ref{p(co)m(co)a-equiv}, we can extend the Diagram~\eqref{diag-pcc->mpc+pma} to the following commutative diagram, under the hypothesis that the Hopf algebra is finite dimensional:

\begin{equation}\label{diag-partial-str}
	\begin{gathered}
		\xymatrix{
			(C, \lambda', H)\ar@{<->}[rrrr]\ar@{<->}[dd] & & & &(C^\ast, \Pmai, H^0)\\
			\\
			(C, \Pmci, H^0) & & & & (C^\ast, \rho', H)\ar@{<->}[uu]\ar@{<->}[llll]
		}
	\end{gathered}
\end{equation}

%%%%%%%%%%%%%%%%%%%%%%%%%%%%%%%%%%%%%%%%%%%%%%%%%%%%%%%%%%%%%%%%%%%%%%%%%%%%%%%%%%%%%%%%%%%%
%%%%%%%%%%%%%%%%%%%%%%%%%%%%%%%%%%%%%%%%%%%%%%%%%%%%%%%%%%%%%%%%%%%%%%%%%%%%%%%%%%%%%%%%%%%%
\subsection{Globalization for partial comodule coalgebras}

In this section, our goal is to introduce the concept of globalization for partial comodules coalgebras.

Let $D$ be a left $H$-comodule coalgebra via $\lambda\colon d \mapsto d^{-1} \otimes d^{0} \in H\otimes D$ and $C$ a subcoalgebra of $D$. In order to induce a coaction on $C$ we can restrict the coaction $\lambda$ to $C$, but in general $\lambda(C)\not\subseteq H \otimes C$. However, if there is a linear map \(\pi\colon D \to C\) we  consider the composite map
\begin{align}
	\lambda'\!\colon C & \longto  H \otimes C \notag\\
	c & \longmapsto  c^{-1} \otimes \pi(c^{0}).\label{def-induced-pcc}
\end{align}

The following result gives us conditions on the map $\pi$ for the above map becomes a partial coaction on $C$.

\begin{prop}[Induced partial comodule coalgebra]\label{ipcc}
	Let $(D, \lambda)$ be a left $H$-comodule coalgebra, $C$ a subcoalgebra of $D$ and $\pi\colon D\to C$ a comultiplicative projection such that
	\begin{equation}
		(I \otimes I \otimes \pi)(I \otimes \lambda\pi) \lambda(c) = (I \otimes I \otimes \pi)(I \otimes \lambda \otimes \Ve \pi)(I \otimes \tau \Delta) \lambda (c), \label{induced-pcc}
	\end{equation}
	for any $c \in C$.
	Then $C$ is a left partial $H$-comodule coalgebra, with structure given by Equation \eqref{def-induced-pcc}.
\end{prop}

\begin{proof}
	First of all, since $\pi$ is a projection from $D$ onto $C$, then $\lambda'$ satisfies the condition (\ref{pcc-1}). In fact, given $c\in C$, we have
	\begin{alignat*}4
		(\Ve\otimes I)\lambda'(c)
			& \EREf \Ve(c\U)\pi(c\Z)\\
			& \EREf \pi(\Ve(c\U)c\Z)\\
			& \REf{cc-1}  \pi(c)\\
			& \EREf c,
	\end{alignat*}
	where the last equality holds since $\pi$ is a projection (and so $\pi(c) = c$, for all $c\in C$).
	
	Now, since $\pi$ is a comultiplicative map (i.e. $\Delta\circ\pi = (\pi\otimes\pi)\circ\Delta$) then it follows that $\lambda'$ satisfies the condition (\ref{pcc-2}). In fact, let $c\in C$, so
	\begin{alignat*}4
		(I \otimes \Delta_C)\lambda' (c)
			& \EREf c^{-1} \otimes \Delta(\pi(c^{0}))\\
			& \EREf c^{-1} \otimes (\pi\otimes\pi)(\Delta(c^{0}))\\
			& \EREf c^{-1} \otimes \pi(c^{0}{}_1)\otimes\pi(c^{0}{}_2)\\
			& \REf{cc-2} c_1{}^{-1}c_2{}^{-1} \otimes \pi(c_1{}^{0})\otimes\pi(c_2{}^{0})\\
			& \EREf (m_H \otimes I \otimes I ) (I \otimes \tau_{C,H} \otimes I)(\lambda' \otimes \lambda')\Delta_C (c).
	\end{alignat*}
	
	Finally, since $\pi$ satisfies Equation \eqref{induced-pcc} we have that, for all $c\in C$
	\begin{alignat*}4
		(I \otimes \lambda')\lambda' (c)
			& \EREf c\Up \otimes c\Zp\Up \otimes c\Zp\Zp\\
			& \EREf c\U \otimes \pi(c\Z)\U \otimes \pi(\pi(c\Z)\Z)\\
			& \REf{induced-pcc}  c\U \otimes c\Z_2\U \otimes \Ve(\pi(c\Z_1)) \pi(c\Z_2\Z)\\
			& \REf{cc-2}  c_1\U c_2\U \otimes c_2\Z\U \otimes \Ve(\pi(c_1\Z)) \pi(c_2\Z\Z)\\
			& \REf{cc-3}  c_1\U \Ve(\pi(c_1\Z)) c_2\U_1 \otimes c_2\U_2 \otimes \pi(c_2\Z)\\
			& \EREf c_1\Up \Ve(c_1\Zp) c_2\Up_1 \otimes c_2\Up_2 \otimes c_2\Zp\\
			& \EREf (m_H \otimes I \otimes I)\{ \nabla \otimes [(\Delta_H \otimes I)\lambda']\}\Delta_C (c).
	\end{alignat*}
	Therefore, $C$ is a left partial $H$-comodule coalgebra.
\end{proof}

\begin{obs}
	One can note that the converse of Proposition \ref{ipcc} is also true. Indeed, supposing $C$ a subcoalgebra of a comodule coalgebra $D$ and $\pi \colon D \to C$ a comultiplicative projection such that $C$ is a partial comodule coalgebra by $\lambda'(c) = c^{-1} \otimes \pi(c^0)$, hence Equation \eqref{induced-pcc} holds.
	
	The proof of the above statement follows straight from the calculations made in Proposition \ref{ipcc}.
\end{obs}

We are now able to define a globalization for a partial comodule coalgebra, as follows.

\begin{dfn}[Globalization for partial comodule coalgebra]\label{gcc}
	Let $C$ be a left partial $H$-comodule coalgebra. A \emph{globalization for $C$} is a triple $(D, \theta, \pi)$, where $D$ is an $H$-comodule coalgebra, $\theta$ is a coalgebra monomorphism from $C$ into $D$ and $\pi$ is a comultiplicative projection from $D$ onto $\theta(C)$, such that the following conditions hold:
	\begin{enumerate}\Not{GCC}
		\item  $x\U \otimes \pi(x\Z)\U \otimes \pi(\pi(x\Z)\Z) =\\ \hspace*{3cm} = x\U \otimes x\Z_2\U \otimes \Ve(\pi(x\Z_1)) \pi(x\Z_2\Z),$ \\ for all $x\in \theta(C)$;\label{gcc-1}
		\item  $\theta$ is an equivalence of partial $H$-comodule coalgebra;\label{gcc-2}
		\item  $D$ is the $H$-comodule coalgebra generated by $\theta(C)$.\label{gcc-3}
	\end{enumerate}
\end{dfn}

\begin{obs}
	The first item in Definition \ref{gcc} tells us that it is possible to define the induced partial comodule coalgebra on $\theta(C)$.
	
	The second one tells us that this induced partial coaction coincides with the original, and this fact is translated in the commutative diagram bellow:
	\begin{equation}\label{diag-coacao-equiv}
		\begin{gathered}
			\xymatrix{
				C\ar[rr]^-{\lambda'}\ar[dd]_-\theta & & H \otimes C\ar[dd]^-{I\otimes\theta}\\
				\\
				\theta(C)\ar[rr]^-{\lambda_\pi} & & H \otimes \theta(C)\ar@{}|-\circlearrowright[uull]
			}
		\end{gathered}
	\end{equation}
	Moreover, the second condition can be seen as
	\begin{equation}
		\theta(c)\U \otimes \pi(\theta(c)\Z) = c\Up \otimes \theta(c\Zp)\label{eq-coacao-equiv},
	\end{equation} for all $c\in C$.
	
	Finally, the last condition of Definition \ref{gcc} tells us that there is no proper subcomodule coalgebra of $D$ containing $\theta(C)$.
\end{obs}

%%%%%%%%%%%%%%%%%%%%%%%%%%%%%%%%%%%%%%%%%%%%%%%%%%%%%%%%%%%%%%%%%%%%%%%%%%%%%%%%%%%%%%%%%%%%
%%%%%%%%%%%%%%%%%%%%%%%%%%%%%%%%%%%%%%%%%%%%%%%%%%%%%%%%%%%%%%%%%%%%%%%%%%%%%%%%%%%%%%%%%%%%
\subsubsection{Correspondence between globalizations}

Given a left partial $H$-comodule coalgebra $C$ we can induce a structure of right partial $H^0$-module coalgebra on $C$ (see Theorem \ref{inducedfromcomcoalg}). It is also true that a (global) $H$-comodule coalgebra induces a (global) $H^0$-module coalgebra. Therefore, given a globalization $(D, \theta, \pi)$ for $C$, one can ask: \emph{Is there some relation between $D$ and $C$ when viewed as  $H^0$-module coalgebras (global and partial, respectively)?} Here we study a little bit more these structures in order to answer this question. The notations previously used are kept.

Let $C$ be a left partial $H$-comodule coalgebra and suppose that $(D, \theta, \pi)$ is a globalization for $C$. From Theorem \ref{inducedfromcomcoalg}, we have that $C$ is a right partial $H^0$-module coalgebra with partial action given by 
\[
	c \Pmc f = f(c\Up) c\Zp,
\]
for all $c\in C$ and $f\in H^0$. Clearly, the same is true for $D$, i.e., we have a structure of $H^0$-module coalgebra on $D$ given by
\[
	d \Mc f = f(d\U) d\Z,
\]
for all $d\in D$ and $f\in H^0$.

\begin{teo}\label{gcc+h0sp=>gmc}
	Let $C$ be a left partial $H$-comodule coalgebra and suppose that $(D, \theta, \pi)$ is a globalization for $C$. If $H^0$ separate points, then $(D, \theta, \pi)$ is also a globalization for $C$, as right partial $H^0$-module coalgebra.
\end{teo}

\newcommand{\Tc}{\theta(c)}

\begin{proof}
	Since $\theta$ is a coalgebra monomorphism from $C$ into $D$ and $\pi$ is a comultiplicative projection from $D$ onto $\theta(C)$, in order to induce a structure of partial $H^0$-module coalgebra on $\theta(C)$ we just need to check that the Equation \eqref{eq:pmc} holds (see Proposition \ref{ipmc}). For this, let $x =  \Tc \Mc g \in \theta(C) \Mc H^0$ and $f \in H^0$, so
	\begin{alignat*}4
		\pi(\pi(x)\Mc f)
			& \EREf \pi(\pi(\Tc\Z g(\Tc\U)) \Mc f) \\
			& \EREf g(\Tc\U) \pi(\pi(\Tc\Z)\Z f(\pi(\Tc\Z)\U)) \\
			& \EREf g(\Tc\U) f(\pi(\Tc\Z)\U) \pi(\pi(\Tc\Z)\Z) \\
			& \REf{induced-pcc} g(\Tc\U) f(\Tc\Z_2\U) \pi(\Tc\Z_2\Z) \Ve\pi(\Tc\Z_1) \\
			& \EREf f((\Tc \Mc g)_2\U) \pi((\Tc \Mc g)_2\Z) \Ve\pi((\Tc \Mc g)_1)\\
			& \EREf f(x_2\U) \pi(x_2\Z) \Ve\pi(x_1)\\
			& \EREf \pi(x_2 \Mc f) \Ve\pi(x_1)\\
			& \EREf \pi(\Ve(\pi(x_1))x_2 \Mc f).
	\end{alignat*}
	Thus, $\theta(C)$ has a structure of partial $H^0$-module coalgebra induced from the structure of module coalgebra of $D$.
	
	Now we show that $\theta$ is a morphism of partial actions. In fact, let $c\in C$, so
	\begin{alignat*}4
		\theta(c) \Pmc f
			& \EREf \pi(\theta(c) \Mc f)\\
			& \EREf	 f(\theta(c)\U)\pi(\theta(c)\Z)\\
			& \REf{eq-coacao-equiv}  f(c\Up)\theta(c\Zp)\\
			& \EREf \theta(f(c\Up)c\Zp)\\
			& \EREf \theta(c\Pmc f).
	\end{alignat*}
	
	Therefore, we just need to show that the (\ref{gmc-3}) in Definition \ref{gmc} holds.
	
	Let $M$ be any $H^0$-submodule coalgebra of $D$ containing $\theta(C)$. We need to show that $M = D$, and for this it is enough to show that $M$ is an $H$-subcomodule coalgebra of $D$.
	
	Take $f\in H^0$, $m\in M$, and consider $\{h_i\}$ a basis of $H$. Let $\{h_i^\ast\}$ be the set contained in $H^\ast$ whose elements are all the dual maps of the $h_i$'s. Then write $\lambda(m) \in H\otimes D$ in terms of the basis of $H$, i.e.,
	\[
		\lambda(m) = \sum\limits_{i=0}^n h_i \otimes m_i,
	\]
	where the $m_i$'s are non-zero elements, at least, in $D$.
	
	Since $D$ is an $H$-comodule, so it is an $H^\ast$-module via the same action of $H^0$. Moreover, the action of $H^0$ on $D$ is a restriction of the action of $H^\ast$. Since $H^0$ separate points, it follows by Jacobson Density Theorem that if $m\in M$ then there exists $\left\{h_{(m)i}^0\in H^0\right\}$ such that $m \Mc h_{(m)i}^0 = m \Mc h_i^\ast$, for each $i$. Thus we have that
	\[
		m \Mc h_{(m)j}^0 = m \Mc h_j^\ast = \sum\limits_{i=0}^n h_j^\ast(h_i) m_i = m_j
	\]
	and so each $m_i$ lies in $M$. Then $M$ is an $H$-subcomodule of $D$, so $M$ is an $H$-subcomodule coalgebra of $D$ containing $\theta(C)$, which implies $M=D$, and the proof is complete.
\end{proof}

%%%%%%%%%%%%%%%%%%%%%%%%%%%%%%%%%%%%%%%%%%%%%%%%%%%%%%%%%%%%%%%%%%%%%%%%%%%%%%%%%%%%%%%%%%%%
%%%%%%%%%%%%%%%%%%%%%%%%%%%%%%%%%%%%%%%%%%%%%%%%%%%%%%%%%%%%%%%%%%%%%%%%%%%%%%%%%%%%%%%%%%%%
\subsubsection{Constructing a globalization}

Now we construct a globalization for a left partial comodule coalgebra $C$ in a special situation. First of all, remember that if $M$ is a right $H^0$-module and $H^0$ separate points, then we have a linear map
\begin{align*}
	\varphi\colon M & \longrightarrow \Hom(H^0, M)\\
	m & \longmapsto \varphi(m)(f) = m \cdot f
\end{align*}
and an injective linear map 
\begin{align*}
	\gamma\colon H\otimes M &\longrightarrow \Hom(H^0, M)\\
	h \otimes m & \longmapsto \gamma(h \otimes m)(f) = f(h) m.
\end{align*}
In the above situation, we say that $M$ is a \emph{rational $H^0$-module} if $\varphi(M) \subseteq \gamma(H \otimes M)$ (cf.~\cite[Definition~2.2.2]{DNR}).

This definition can be seen in the following commutative diagram:
\begin{equation}
	\begin{gathered}
		\xymatrix{
			& \Hom(H^0, M) & \\
			\\
			M\ar[ruu]^-\varphi \ar@{-->}[rr]^-{\lambda} & & H \otimes M\ar@{_{(}->}[uul]_-\gamma
		}
	\end{gathered}
\end{equation}

Notice that, given a rational $H^0$-module $M$ we have on it a structure of $H$-comodule via $\lambda\colon M \to H \otimes M$ satisfying, for any $m\in M$
\begin{equation}\label{eq-h0rat}
	\lambda(m) = \sum h_i \otimes m_i \iff m \cdot f = \sum f(h_i) m_i,\ \text{for all}\,\ f\in H^0.
\end{equation}

From Theorem \ref{gcc+h0sp=>gmc}, it follows that exists a natural way to look for a globalization for a partial coaction of $H$ on $C$, i.e., to get a globalization for $C$ we should see it as a partial module coalgebra and then consider its standard globalization as module, under the hypothesis that $H^0$ separate points.

Thus, given a left partial $H$-comodule coalgebra $(C, \lambda')$ we have from Theorem \ref{teo-gmc} that $(C\otimes H^0, \theta, \pi)$ is a globalization for $C$ as right partial $H^0$-module coalgebra.

We desire $C \otimes H^0$ to be a globalization for $C$ as partial comodule coalgebra, but, in general, it is not even an $H$-comodule coalgebra.

In order to overcome this problem we will suppose that $C \otimes H^0$ is a rational $H^0$-module and, therefore, we have that $C \otimes H^0$ is an $H$-comodule with coaction satisfying Equation \eqref{eq-h0rat}, i.e., the following holds
\begin{equation}\label{eq-CxH0-h0rat}
	\lambda(c \otimes f) = \sum h_i \otimes c _i \otimes f_i \iff c \otimes (f \ast g) = \sum g(h_i) c _i \otimes f_i,
\end{equation}
for any $c\otimes f \in C \otimes H^0$ and $g\in H^0$

Therefore, by Theorem \ref{h0seppontops-pcc-pmci-pmai}, $C \otimes H^0$ is an $H$-comodule coalgebra. Now we are in position to show that $C \otimes H^0$ is a globalization for $C$ as partial comodule colagebra, as follows.

\begin{teo}\label{teo-gcc}
	Let $C$ be a left partial $H$-comodule coalgebra. With the above notations, if $C \otimes H^0$ is a rational $H^0$-module and $H^0$ separate points, then $(C \otimes H^0, \theta, \pi)$ is a globalization for $C$.
\end{teo}

\begin{proof}
	By the above discussed, $C \otimes H^0$ is an $H$-comodule coalgebra, $\theta\colon C \to C\otimes H^0$ is a coalgebra map and ${\pi\colon C \otimes H^0 \twoheadrightarrow \theta(C)}$ is a comultiplicative projection. 
	
	Now, we can show directly that the conditions (\ref{gcc-1})-(\ref{gcc-3}) hold. Let $x \in C \otimes H^0$, and $f,g\in H^0$, so
	
	\begin{alignat*}4
		\mathrlap{\hspace*{-2cm}(f \otimes g \otimes I)[x^{-1} \otimes \pi(x^{0})\U \otimes \pi(\pi(x^{0})\Z)] = } \\
		\hspace*{2cm} & \EREf f(x^{-1}) g(\pi(x^{0})\U) \pi(\pi(x^{0})\Z) \\
			& \EREf f(x^{-1}) \pi(g(\pi(x^{0})\U) \pi(x^{0})\Z)  \\ %\pi[g(\pi(c \otimes f)\U) \pi(c \otimes f)\Z] \\
			& \EREf f(x^{-1})  \pi[\pi(x^{0}) \Mc g] \\
			& \EREf \pi[\pi(x^{0} f(x^{-1})) \Mc g] \\
			& \EREf \pi[\pi(x \Mc f) \Mc g] \\
			& \REf{gmc-2} \pi[\Ve \pi((x \Mc f)_1) (x \Mc f)_2 \Mc g] \\
			& \EREf \pi[\Ve \pi(x_1 \Mc f_1) (x_2 \Mc f_2) \Mc g]\\
			& \EREf \pi[\Ve \pi(x_1{}^0 f_1(x_1{}^{-1})) (x_2{}^0 f_2(x_2{}^{-1}) \Mc g)]\\
			& \EREf f_1(x_1{}^{-1}) f_2(x_2{}^{-1}) \Ve \pi(x_1{}^0 ) \pi(x_2{}^0  \Mc g)\\
			& \EREf f(x_1{}^{-1} x_2{}^{-1}) \Ve \pi(x_1{}^0 ) \pi(x_2{}^0  \Mc g) \\
			& \REf{pcc-2} f(x^{-1}) \Ve \pi(x^{0}{}_1 ) \pi (x^{0}{}_2  \Mc g)\\
			& \EREf f(x^{-1}) \Ve \pi(x^{0}{}_1 ) \pi (x^{0}{}_2{}^{0}   g(x^{0}{}_2{}^{-1}))\\
			& \EREf f(x^{-1}) g(x^{0}{}_2{}^{-1}) \Ve \pi(x^{0}{}_1 ) \pi (x^{0}{}_2{}^{0})\\
			& \EREf (f \otimes g \otimes I) (x^{-1} \otimes x^{0}{}_2{}^{-1} \otimes \Ve \pi(x^{0}{}_1 ) \pi (x^{0}{}_2{}^{0})) .
	\end{alignat*}
	
	Since $H^0$ separate points, the condition (\ref{gcc-1}) is satisfied.
	
	To prove the condition (\ref{gcc-2}) take $c\in C$ and note that
	\begin{alignat*}4
		(g \otimes I)[\theta(c)\U \otimes \pi(\theta(c)\Z)]
			& \EREf g(\theta(c)\U) \pi(\theta(c)\Z)\\
			& \EREf \pi(g(\theta(c)\U) \theta(c)\Z)\\
			& \EREf \pi(\theta(c) \Mc g)\\
			& \EREf \theta(c \Pmc g)\\
			& \EREf g(c\Up) \theta(c\Zp)\\
			& \EREf (g \otimes I)[c\Up \otimes \theta(c\Zp)].
	\end{alignat*}
	Since $H^0$ separate points, thus $\theta$ is an equivalence of partial coactions.
	
	Finally, to show that $C \otimes H^0$ is generated by $\theta(C)$, consider a subcomodule coalgebra $M$ of $C \otimes H^0$ containing $\theta(C)$. By Theorem \ref{inducedfromcomcoalg}, $M$ is an $H^0$-submodule coalgebra of $C \otimes H^0$ containing $\theta(C)$. Thus, it follows from condition (\ref{gmc-3}) that $M = C \otimes H^0$.
	
	Therefore $C \otimes H^0$ is a globalization for $C$ as a partial $H$-comodule coalgebra.
\end{proof}

The globalization above constructed is called the \emph{standard globalization} for a partial comodule coalgebra.

\begin{obs}
	When the Hopf algebra is finite dimensional, $H^0 = H^\ast$ and so it separate points. Furthermore, in this case $C \otimes H^\ast$ is a rational $H^\ast$-module with coaction given by
	\[
		\lambda\colon c\otimes f \longmapsto\sum\limits_{i = 1}^n h_i \otimes c \otimes f\ast h_i^\ast,
	\]
	where $\{ h_i, h_i^\ast \}_{i=1}^n$ is a dual basis for $H$ and $H^\ast\!$.
\end{obs}

%%%%%%%%%%%%%%%%%%%%%%%%%%%%%%%%%%%%%%%%%%%%%%%%%%%%%%%%%%%%%%%%%%%%%%%%%%%%%%%%%%%%%%%%%%%%
%%%%%%%%%%%%%%%%%%%%%                                %%%%%%%%%%%%%%%%%%%%%%%%%%%%%%%%%%%%%%%
%%%%%%%%%%%%%%%%%%%%%         AGRADECIMENTOS         %%%%%%%%%%%%%%%%%%%%%%%%%%%%%%%%%%%%%%%
%%%%%%%%%%%%%%%%%%%%%                                %%%%%%%%%%%%%%%%%%%%%%%%%%%%%%%%%%%%%%%
%%%%%%%%%%%%%%%%%%%%%%%%%%%%%%%%%%%%%%%%%%%%%%%%%%%%%%%%%%%%%%%%%%%%%%%%%%%%%%%%%%%%%%%%%%%%
%\section*{Acknowledgments}

%The authors would like to thank Antonio Paques and Alveri Sant'Ana, whose comments, corrections and suggestions were very useful to improve the manuscript.

%%%%%%%%%%%%%%%%%%%%%%%%%%%%%%%%%%%%%%%%%%%%%%%%%%%%%%%%%%%%%%%%%%%%%%%%%%%%%%%%%%%%%%%%%%%%
%%%%%%%%%%%%%%%%%%%%%                                %%%%%%%%%%%%%%%%%%%%%%%%%%%%%%%%%%%%%%%
%%%%%%%%%%%%%%%%%%%%%           BIBLIOGRAFIA          %%%%%%%%%%%%%%%%%%%%%%%%%%%%%%%%%%%%%%%
%%%%%%%%%%%%%%%%%%%%%                                %%%%%%%%%%%%%%%%%%%%%%%%%%%%%%%%%%%%%%%
%%%%%%%%%%%%%%%%%%%%%%%%%%%%%%%%%%%%%%%%%%%%%%%%%%%%%%%%%%%%%%%%%%%%%%%%%%%%%%%%%%%%%%%%%%%%

\end{document}